\documentclass[preprint]{article}
\usepackage{algorithm2e}
\usepackage{amsthm}
\usepackage{graphicx}
\usepackage[colorlinks,
            linkcolor=blue,
            anchorcolor=blue,
            citecolor=blue
            ]{hyperref}
\usepackage[a4paper]{geometry}
\usepackage{setspace}
\usepackage{array}
\usepackage{booktabs}
\newtheorem{definition}{Definition}
\newtheorem{theorem}{Theorem}
\newtheorem{lemma}{Lemma}

\newtheorem{property}{Property}

\begin{document}
\begin{spacing}{1.15}

\title{Independent Italian Domination on Block Graphs \vspace{-2em} }
\date{}
\maketitle
\begin{center}
\author{Decheng Wei, Changhong Lu$^{*}$ \\
\normalsize{School of Mathematical Sciences} \\
\normalsize{East China Normal University} \\
\normalsize{Shanghai 200241,  P. R. China}}
\end{center}

\begin{abstract}
   Given a graph $G=(V,E)$, $f:V \rightarrow \{0,1,2 \}$ is the Italian dominating function of $G$ if $f$ satisfies $\sum_{u \in N(v)}f(u) \geq 2$ when $f(v)=0$. Denote $w(f)=\sum_{v \in V}f(v)$ as the weight of $f$.
   Let $V_i=\{v:f(v)=i\},i=0,1,2$, we call $f$ the independent Italian dominating function if $V_1 \cup V_2$ is an independent set. The independent Italian domination number of $G$ is the minimum weight of independent Italian dominating function $f$, denoted by $i_{I}(G)$. We equivalently transform the independent domination problem of the connected block graph $G$ to the induced independent domination problem of its block-cutpoint graph $T$, then a linear time algorithm is given to find $i_{I}(G)$ of any connected block graph $G$ based on dynamic programming.

\end{abstract}
\footnotetext{\noindent E-mail:ecnuwdc@163.com (D. Wei); chlu@math.ecnu.edu.cn (C. Lu)}
%\footnote{ E-mail:ecnuwdc@163.com (D. Wei); chlu@math.ecnu.edu.cn (C. Lu)}

\section{Introduction}
\qquad  In the 3rd century, when Rome dominated Europe, it was able to deploy 50 legions throughout the empire, securing even the furthermost areas. By the following century, Roman's forces had diminished to just 25 legions. Emperor Constantine's problem: How to station legions in sufficient strength to protect the most forward positions of the empire without abandoning the core, namely Rome. He devised a new defensive strategy to cope with Rome's reduced power \cite{Revelle 1, Revelle 2}. Cockayne, Dreyer and Hedetniemi proposed the Roman dominating function in 2000 based on Constantine's strategy \cite{Roman dominating function}. Given a graph $G=(V,E)$, a function $f: V(G) \rightarrow \{0, 1, 2 \}$ is a Roman dominating function of $G$ if there exists $u \in N(v)$ satisfying $f(u) = 2$ when $f(v)=0$. Brešar, Henning and Rall proposed k-rainbow dominating function in 2005 \cite{rainbow domination}. A function $f: V(G) \rightarrow \{0, 1, 2\}$ is a $\{2\}$-dominating function if $\sum_{v \in N[u]}f(v) \geq 2$ for any $u \in V(G)$.  A Roman k-dominating function \cite{Roman k-domination} on $G$ is a function $f: V(G) \rightarrow \{0, 1, 2\}$ such that every vertex $u$ for which $f(u)=0$ is adjacent to at least $k$ vertices $v_1, v_2, ... , v_k$ with $f(v_i)=2$ for $i=1, 2, ... , k$. Chellali, Haynes and Hedetniemi proposed Roman $\{2\}$-dominating function in 2015 \cite{Roman2 domination}.  Henning and Klostermeyer denoted the Roman $\{2\}$-dominating function as Italian dominating function to simplify the description \cite{Italian domination}. A function $f: V(G) \rightarrow \{0,1,2 \}$ is an Italian dominating function of $G$ if there exists $u \in N(v)$ satisfying $\sum_{u \in N(v)}f(u) \geq 2$  when $f(v)=0$.  $w(f)= \sum_{v \in V}{f(v)}$ is the weight of $f$.  The Italian domination number of $G$ is the minimum weight of an Italian dominating function, denoted as $\gamma_{I}(G)$ and the corresponding function is called the minimum Italian dominating function. Let $V_i=\{v:f(v)=i\},i=0,1,2$, we call $f$ the independent Italian dominating function if $V_1 \cup V_2$ is an independent set. The independent Italian domination number of $G$ is the minimum weight of its independent Italian dominating function, denoted by $i_{I}(G)$ and the corresponding function is called the minimum independent Italian dominating function, denoted as $i_{I}$-function. % Let $G=(V,E)$ be a graph of order $|V|=n$.
 For any vertex $v \in V$, the open neighborhood of $v$ is the set $N(v)=\{u \in V: uv \in E \}$ and the closed neighbourhood is the set $N[v]=N(v) \cup \{v\}$. Generally, $N_{G}(v)$ and $N_{G}[v]$ represents the open neighborhood and closed neighborhood  of $v$ in $G$ respectively. We set $N^{2}(v) = \cup_{u \in N(v)}N(u)$ and  $N_{G}^{2}(v) = \cup_{u \in N_{G}(v)}N_{G}(u)$.
 A cut vertex in a connected graph is a vertex whose deletion breaks the graph into two or more parts. The subgraph $H$ of $G$ is a block if it is maximal and has no cut-vertex. We call the vertex $v$ an uncut-vertex if $v$ is in a block and $v$ is not a cut-vertex. If the subgraph of a graph is an complete graph, then the subgraph is a clique. $G$ is a block graph if each block of $G$ is a clique. we define the block-cutpoint graph of a block graph as follows.

\begin{definition}\label{definition1}
       Given a block graph $G$, the block-cutpoint graph of $G$ is a bipartite graph $T$ in which one partite set consists of the cut-vertices of $G$ and the other one consists of vertex $b_i$  corresponding to each block $B_i$ of $ G$.  $vb_i \in E(T)$  if and only if $v \in B_i$ where $v$ is a cut-vertex and $E(T)$ is the set of edges of $T$, calling $b_i$ the block-vertex of $T$.
\end{definition}

\begin{lemma}\label{lemma1}
      The block-cutpoint graph $T$ is a tree when the corresponding graph $G$ is a connected block graph with at least one cut-vertex.
\end{lemma}
\begin{proof}
     It is clear that two blocks of a graph share at most one vertex, then the block-cutpoint graph $T$ has no cycles. As $G$ is connected, $T$ is also connected, So $T$ is a tree. We can construct the block-cutpoint graph $T$ in linear time with Depth-First Search,
     see \cite{Createblock-cutpointgraph}.
\end{proof}

It is obvious that $G$ is a connected block graph without cut-vertices if and only if $G$ is a complete graph and  $i_{I}(K_1)=1$, $i_{I}(K_n)=2 (n\geq 2)$.
If $G$ is a block graph with two connected components $G_1$ and $G_2$, then $i_{I}(G)=i_{I}(G_1) + i_{I}(G_2)$. Therefore, we just need to consider $G$ is a connected block graph with at least one cut-vertex. Without special illustration, the block graphs being referred to in this paper are all  connected block graph with at least one cut-vertex, so the corresponding block-cutpoint graph is a tree. We can obtain the independent Italian domination number of a connected tree through dynamic programming. In this paper, our main task is to find the independent Italian domination number of a block graph. As the block-cutpoint graph of a connected block graph $G$ is a tree, we design a linear time algorithm to output the independent Italian domination number of $G$ based on dynamic programming.

\section{Independent Italian domination in block graphs}

\qquad we classify the block of a graph into three types by the number of vertices and cut-vertices in the block.
\begin{definition} \label{definition2}
     Given a block graph $G$ and its block-cutpoint graph $T$. $B$ is a block of $G$ and $C$ is the set of cut-vertices of $B$. Let $b$ be the corresponding block-vertex of $B$ in $T$.  $B$ $(b)$ is a block $($block-vertex$)$ of type 0 in $G$ $(T)$ if $|B|=|C|$ . $B$ $(b)$ is a block $($block-vertex$)$ of type 1 in $G$ $(T)$ if $|B|=|C|+1$. $B$ $(b)$ is a block $($block-vertex$)$ of type 2 in $G$ $(T)$ if $|B| \geq |C|+2$.
\end{definition}

\begin{figure} \label{figure}
  \centering
  \includegraphics[width=6in]{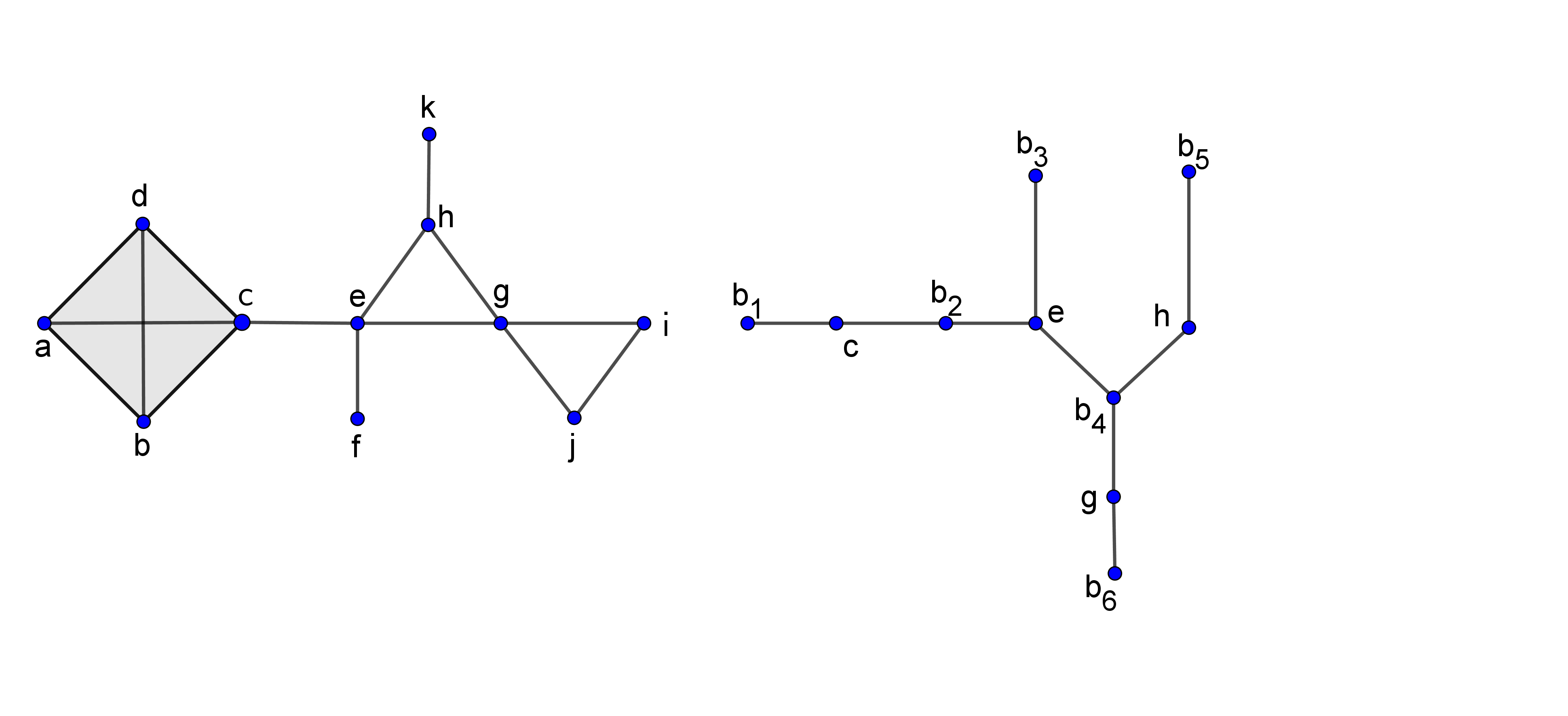}
  \caption{a bock graph and its corresponding block-cutpoint graph}
\end{figure}

\indent In order to express concisely, we denote the block of type 0 as block0  and the block-vertex of type 0 as block0-vertex. Getting block1, block1-vertex,block2 and block2-vertex respectively in the same way. In figure \textcolor{blue}{1}, $b_2$ and $b_4$ are block0-vertices. $b_3$ and $b_5$ are block1-vertices. $b_1$ and $b_6$ are block2-vertices. $c,e,h$ and $g$ are cut-vertices in both graphs. We want to set up an equivalent relationship between a block graph $G$ and its block-cutpoint graph $T$ and then we can transfer the problem of the block graph $G$ to its block-cutpoint graph $T$. Actually, we are trying to  transfer the independent Italian domination problem from a block graph to a tree.  We need to define a new induced function $f_*$ of $T$ which is equivalent to the independent Italian dominating function $f$ of $G$.

\begin{definition}\label{definition3}
      Given a block graph $G$ and its block-cutpoint graph $T$, let $B$ be an arbitrary block of $G$ and $C$ be the set of cut-vertices of $B$ and $b$ be the corresponding block-vertex of $B$ in $T$. $f$ is a function of $G$. $f_*$ is an induced function of $ T$ induced by $f$ if it satisfies $f_*(v)=f(v)$ for each cut-vertex $v$ and $f_*(b) = f(B)-f(C)$ for each block $B$. If $f$ is an Italian dominating function, then the corresponding $f_*$ is an induced Italian  dominating function. If $f$ is an independent Italian dominating function, then the corresponding $f_*$ is an induced independent Italian  dominating function (IIIDF).
\end{definition}

     $w(f_*)=\sum_{v \in V(T)}f_*(v)$ is the weight of $f_*$. It is obvious that $w(f) = w(f_*)$ according to the definition of $f_*$. The induced independent Italian domination number of $T$ is the minimum weight of $f_*$, of which $f$ is the independent Italian dominating function. Denote the induced independent Italian domination number of $T$ as $i_{I}^{*}(T)$. We call the function $f_*$ satisfying $w(f_*) = i_{I}^{*}(G)$ the minimum induced Italian dominating function of $T$, denoted as $i_{I}^{*}$-function.

\begin{lemma}\label{corollary1}
     Given a  block graph $G$ and its block-cutpoint graph $T$, $i_{I}(G)=i_{I}^{*}(T)$.
\end{lemma}

\begin{proof}
     Let $f$ be an independent Italian dominating function of $G$ and $f_*$ be the induced independent Italian dominating function of $T$ induced by $f$, then $w(f)=w(f_*)$. Obviously, $w(f_*) \geq i_{I}^{*}(T)$. However, if $f_{*}^{'}$ is an induced independent Italian dominating function of $T$ with the weight $ w(f_*^{'})=i_{I}^{*}(T) $ and $f_*^{'}$ is induced by the independent Italian dominating function $f^{'}$ of $G$, then $w(f_*^{'})=w(f^{'})$, so $w(f^{'})=i_{I}^{*}(T)$.
     As $i_{I}(G) \leq w(f^{'})$, then $i_{I}(G) \leq i_{I}^{*}(T)$. We can get $i_{I}(G) \geq i_{I}^{*}(T)$ in the same way. Hence, $i_{I}(G) = i_{I}^{*}(T)$.
\end{proof}

We have set up an relationship between  the block graph $G$ and the block-cutpoint graph $T$. It seems that we have already transferred the Italian domination problem from block graph to its block-cutpoint graph successfully, however, there still remains one problem to research.  Given a function $f_*$ of a block-cutpoint graph $T$, how can we distinguish whether it is an induced independent Italian dominating function or not? This problem will be solved in the following.
Let $G$ be an arbitrary block graph and $T$ be the block-cutpoint graph of $G$. $f$ is an independent Italian dominating function of $G$ and $f_*$ is the corresponding induced independent Italian dominating function  of $T$ induced by $f$. We have the following results.

\begin{theorem}\label{block2-uncut-vertex}
    Let $B$ be an arbitrary block of $G$. $\forall  v \in B $  and $v$ is a cut-vertex, then $f(v) \in \{0, 2\}$ when $B$ is block1 or block2.  $\forall v \in B$ and $v$ is an uncut-vertex, then $f(v) \in \{0, 2\}$ and there exists at most one vertex $v$ such that $f(v) = 2$ when $B$ is block2.
\end{theorem}
\begin{proof}
   $\forall v \in  B$ and $B$ is a block1 or block2, assuming that $v$ is a cut-vertex and $f(v)=1$. Considering the independence of $f$, we have $f(u) = 0$ if  $u \in B$ and $u$ is an uncut-vertex. Obviously, $\forall w \in N(v), f(w)=0$, therefore, we get that $\sum_{u^{'} \in N(u)} f(u^{'})=1$, contradiction.  If $v$ is an uncut-vertex in $B$ with $f(v)=1$ and $B$ is a block2, then $f(u)=0, \forall u \in N(v)$ for the independence of $f$. There is at least one uncut-vertex $w^{'} \in N(v)$ such that $f(w^{'}) = 0$, hence we have $\sum_{w^{''} \in N(w^{'})}f(w^{''})=1$, contradiction. Therefore, if $v$ is an uncut-vertex in $B$ and $B$ is a block2, then $f(v) \in \{0, 2\}$. Obviously, there is at most one vertex $v \in B$ such that $f(v)=2$ for the independence of $f$.
\end{proof}
\begin{theorem}\label{ExistPosiblilityOfMiDRDF-02}
     There exists an independent Italian dominating function $f$ of $G$ such that
     $f(v) \in \{0, 1 \}$ where $w(f) = i_{I}(G)$ and $v$ is an uncut-vertex in a block1 of $G$.
\end{theorem}
\begin{proof}
    Assuming that $f$ is an independent Italian dominating function of $G$ with $w(f) = i_{I}(G)$ and $B$ is an arbitrary block1 of $G$. Let $v$ be the only uncut-vertex of $B$ such that $f(v)=2$, then $\forall u \in N_{G}(v)$, $f(u) = 0$. Let $G^{'} = G - B$, if $\exists u_0 \in N_{G}(v)$ satisfying $\forall u_{0}^{'} \in N_{G^{'}}(u_0)$, $f(u_{0}^{'}) = 0$, then let $f(v) = 0$ and $f(u_0) = 2$.
    If $\forall u_1 \in N_{G}(v)$, $\exists u_{1}^{'} \in N_{G^{'}}(u_1)$ satisfying $f(u_1^{'}) \neq 0$, then let $f^{'}(v) = 1$ and $f^{'}|_{G-v} = f|_{G-v}$. Obviously, $f^{'}$ is also an independent Italian dominating function, however, $w(f) = w(f^{'}) +1$, then $w(f^{'}) < i_{I}(G)$, contradiction. Therefore, $f(v) \neq 2$ and $f(v) \in \{0, 1 \}$.
\end{proof}

\begin{property} \label{PropertyFiveProperty}
\begin{spacing}{1.27}
     Five properties of an induced independent Italian dominating function $f_*$ of $T$ will be given below: \\
     \noindent(1) If $b$ is  a block0-vertex  of $T$, then $f_{*}(b) = 0$. \\
     \noindent(2) If $b$ is a block2-vertex of $T$, then $f_{*}(b) \in \{0, 2 \}$. \\
     \noindent(3) If $b$ is a block1-vertex or block2-vertex with $f*(b)=0$, then there exists only one vertex $u \in N_{T}(b)$ such that $f_{*}(u) = 2$. \\
     \noindent(4) If $v$ is a cut-vertex of $T$ with $f*(v)=0$, then there exists $u \in N_{T}^{2}(v)$ such that $f_*(u)= 2$ or exists $u_1,u_2 \in N_{T}^{2}(v)$ such that $f_*(u_1)=f_*(u_2)=1$ where $u_1$ and $u_2$ are not adjacent with the same block-vertex.  \\
     \noindent(5) If $v$ is a cut-vertex of $T$ with $f_*(v)\neq 0$, then  $ \forall w \in N_{T}^{2}(v), f_*(w)=0$. If $b$ is a block1-vertex with $f_*(b) \neq 0$, then $\forall w \in N_{T}(b)$, $f_*(w)=0$. If $b$ is a block2-vertex with $f_*(b) \neq 0$, then $\forall w \in N_{T}(b)$, $f_*(w)=0$  and there is only one uncut-vertex $u$ in the corresponding block2 $B$ of $G$ such that $f(u)\neq 0$.
\end{spacing}
\end{property}
\begin{proof}
      Let $b$ be an arbitrary block-vertex of $T$ and $B$ is the corresponding block of $b$ in $G$, and $C$ is the set of cut-vertices of $B$. Getting $f_*(b)=f(B)-f(C)$ according to the definition \ref{definition3}. Proving the properties in order: \\
     \noindent(\romannumeral1)If $b$ is a block0-vertex, then $|B|=|C|$. Therefore, $f(B)-f(C)=0$ and $f_*(b)=0$. \\
     \noindent(\romannumeral2)If $b$ is a block2-vertex, then$f(B)-f(C) \in \{0,2\}$. Therefore, $f_*(b) \in \{0,2 \}$.\\
     \noindent(\romannumeral3) If $b$ is a block1-vertex or block2-vertex of $T$ with $f_*(b)=0$, then $B$ is the block1 or block2 of $G$ with $f(B)-f(C)=0$. We can easily find $f(v)=0$ for any uncut-vertex $v \in B$. Hence, there exists only one cut-vertex $u \in B$ such that $f(u)=2$. Since $u$ is also the cut-vertex of the block-cutpoint graph $T$, then $f_*(u)=f(u)=2$. Therefore, $\exists ! u \in N_{T}(b)$ such that $f_*(u)=2$. \\
     \noindent(\romannumeral4)If $v$ is a cut-vertex of $T$ with $f_*(v)=0$, then $v$ is also the cut-vertex of $G$ with $f(v)=0$. Since $f$ is an iDRDF of $G$, therefore, there exists $u^{'} \in N_{G}(v)$ such that $f(u^{'})=2$ or exists $u_{1}^{'},u_{2}^{'} \in N_{G}(v)$ such that  $f(u_1^{'})=f(u_2^{'})=1 $ where $u_1^{'}$ and $u_2^{'}$ are not in the same block of $G$. We can find that there exists $u \in N_{T}^{2}(v)$ such that $f_*(u)=2$  if $\exists u^{'} \in N_{G}(v)$ with $f(u^{'})=2$ and there exists $u_1, u_2 \in N_{T}^{2}(v)$ such that $f_*(u_1)=f_*(u_2)=2$ where $u_1$ and $u_2$ are not adjacent with the same block-vertex if $\exists u_1^{'},u_2^{'} \in N_{G}(v)$ with $f(u_1^{'})=f(u_2^{'})=1$ where $u_1^{'}$ and $u_2^{'}$ are not in the same block. \\
     \noindent(\romannumeral5)If $v$ is a cut-vertex of $T$ with $f_*(v) \neq 0$, then $v$ is also the cut-vertex of $G$ with $f(v)=f_*(v) \neq 0$. Since $f$ is an independent  Italian dominating function of $G$, then $\forall w^{'} \in N_{G}(v), f(w^{'})=0 $ and this shows that $ \forall w \in N_{T}^{2}(v), f_*(w)=0$. If $b$ is a block1-vertex of $T$ with $f_*(b) \neq 0$, then $f(B)-f(C)=f_*(b) \neq 0$. We can find that there exists only one uncut-vertex $u \in B$ such that $f(u) \neq 0$, therefore, $\forall w^{'} \in N_{G}(u), f(w^{'})=0$. Hence, $\forall w \in N_{T}(b), f_*(w)=0$. If $b$ is a bock2-vertex of $T$ with $f_*(b) \neq 0$, then $f(B)-f(C) \neq 0$. Hence, there exists $v\in B$ such that $f(v) \neq 0$ and $v$ is unique for the independence of $f$. We can find that $\forall w^{'} \in N_{G}(v), f(w^{'})=0$, therefore $\forall w \in N_{T}(b), f_*(w)=0$.
\end{proof}

     With the accomplishment of  the proof of the five properties in property \ref{PropertyFiveProperty}, we can solve the problem that how to distinguish whether a function $f_*$ of the block-cutpoint graph $T$ is an induced independent Italian dominating function or not.  The problem can be  solved based on the new theorem below. It is after the proof of the new theorem that we can design the linear time algorithm to output the Italian domination number $i_{I}(G)$ of any connected block graph $G$. \\

\begin{theorem} \label{ThmEquiRlationship}
$T$ is the block-cutpoint graph of a block graph $G$. Given a function $f:V(G) \rightarrow \{0,1,2\}$,  $f_*$ is the function  of $T$ induced by $f$. $f$ is an independent Italian dominating function of $G$ such that $f(v) \in \{0,1\}$ for any uncut-vertex $v$ in an arbitrary block1 $B$ of $G$  if and only if  $f_*$ satisfies the 5 properties of  property \ref{PropertyFiveProperty} and $f_*(b) \in \{0, 1 \}$ for any block1-vertex $b\in T$.
\end{theorem}
\begin{proof}
     \noindent(\uppercase\expandafter{\romannumeral1})Necessity:
     If $f$ is an independent Italian dominating function  of $G$, then $f_*$ is an induced independent Italian dominating function of $T$. Therefore, the five properties of property \ref{PropertyFiveProperty} are obviously correct. $f(v) \in \{0,1\}$ for $v \in B$ where $B$ is an arbitrary block1 of $G$ and $v$ is an uncut-vertex. Let $b$ be the block1-vertex of $T$ corresponding to $B$.  Since there exists only one  uncut-vertex in block1, $f(B)-f(C)=f(v)$ and $f_*(b)=f(B)-f(C)=f(v)$. Hence, $f_*(b) \in \{0,1 \}$. The arbitrariness of $B\subseteq G $ can promise the arbitrariness of $b \in T$.

     \noindent(\uppercase\expandafter{\romannumeral2})Sufficiency:
     If $f_*$ is a function of $T$ such that $f_*(b)\in \{0, 1\}$ for any block1-vertex $b\in T$ and $f_*$ satisfies property \ref{PropertyFiveProperty}, we just need to prove that the corresponding function $f$ is an induced independent Italian dominating function of $G$ and $\forall v \in B, f(v)\in \{0,1\}$ where $B$ is an arbitrary block1 and $v$ is an uncut-vertex.

     (\romannumeral1)Proving $f$ is an independent function of $G$: Let $v$ be a cut-vertex of $G$ with $f(v) \neq 0$, then $v$ is also the cut-vertex of $T$ with $f_*(v)=f(v) \neq 0$. Then $\forall w\in N_{T}^{2}(v), f_*(w)=0$  according to the fifth property of property \ref{PropertyFiveProperty}. Therefore, $\forall w^{'} \in N_{G}(v), f(w^{'})=0$. ~ Let $u \in B$ be an uncut-vertex of $G$ with $f(u) \neq 0$. If $B$ is a block1 and $b$ is the corresponding block1-vertex, then $f_*(b)=f(u) \neq 0$. Hence, $\forall w_1 \in N_{T}(b), f_*(w_1)=0$ according the fifth property of property \ref{PropertyFiveProperty} where $w_1$ is the cut-vertex of $T$ and $G$, so $f(w_1^{'})=0$  $\forall w_1^{'} \in N_{G}(u)$ since $u$ is the only one uncut-vertex of $B$. If $B$ is a block2 and the corresponding $b$ is a block2-vertex of $T$ with $f_*(b) \neq 0$, then $\forall w_2\in N_T(b), f_*(w_2)=0$ and there exists only one uncut-vertex $v \in B$ such that $f(v) \neq 0$ according to the fifth property of property \ref{PropertyFiveProperty}, hence $\forall w_2^{'}\in N_{G}(u), f(w_2^{'})=0$. Therefore, $f$ is an independent function of $G$.

     (\romannumeral2) Proving $f$ is an Italian dominating function of $G$ in this part: Let $v$ be a cut-vertex of $G$ with $f(v)=0$, then $v$ is also the cut-vertex of $T$ with $f_*(v)=f(v)=0$. There exists $u\in N_{T}^{2}(v)$ such that $f_*(u)=2$ or exists $u_1,u_2 \in N_{T}^{2}(v)$ such that $f_*(u_1)=f_*(u_2)=1$ where $u_1$ and $u_2$ are not adjacent with the same block-vertex according to the forth property of
     property \ref{PropertyFiveProperty}. Therefore, there exists $u^{'} \in N_{G}(v)$ such that $f(u^{'})=2$ or exists $u_1^{'},u_2^{'} \in N_{G}(v)$ such that $f(u_1^{'})=f(u_2^{'})=1$ where $u_1^{'}$ and $u_2^{'}$ are not in the same block of $G$. ~  Let $u \in B$ be an uncut-vertex of $G$ with $f(u)=0$. If $B$ is a block1, then $f_*(b)=f(u)=0$. There exists only one vertex $v^{'} \in N_{T}(b)$ such that $f_*(v^{'})=2$ according to the third property of property \ref{PropertyFiveProperty}. Since $v^{'}$ is a cut-vertex of $T$, then $v^{'}$ is  also a cut-vertex of $G$ and $v^{'} \in N_{G}(u)$, getting $f(v^{'})=f_*(v^{'})=2$. If $B$ is a block2, then $f(w) \in \{0,2\}$ for any uncut-vertex $w \in N_{G}(u)$ according to theorem \ref{block2-uncut-vertex}. $f_*(b)=0$ when $f(w)=0$ for any uncut-vertex $w \in N_{G}(u)$, then there exists $v^{''} \in N_{T}(b)$ such that $f_*(v^{''})=2$ where $v^{''}$ is a cut-vertex of $T$ according to the third property of property \ref{PropertyFiveProperty}, so $v^{''}$ is also a cut-vertex of $G$ and $v^{''} \in N_{G}(u)$, then $f(v^{''})=f_*(v^{''})=2$. Otherwise, there exists an uncut-vertex $w_0 \in N_{G}(u)$ such that $f(w_0)=2$, hence $\sum_{v^{''} \in N_{G}(u)}f(v^{''})=2$. Therefore, $f$ is an Italian dominating function of $G$.

     (\romannumeral3)Proving that $f(v) \in \{0,1\}$ for any uncut-vertex $v$ of an arbitrary block1 $B_1$ of $G$: Since $f_*(b) \in \{0,1 \}$ for any block1-vertex $b$ of $T$  and there exists only one uncut-vertex $v$ in block1 $B$, then $f(v)=f(B)-f(C)=f_*(b)$, therefore $f(v) \in \{0,1\}$. \\
\end{proof}

  Theorem \ref{ThmEquiRlationship} set up an equivalent relationship between a block graph $G$ and its block-cutpoint graph $T$, then  we can transfer the independent Italian domination problem of $G$ to the induced independent Italian domination problem of $T$. In order to find the Italian domination number $i_{I}(G)$, we just need to find the corresponding induced Italian domination number $i_{I}^{*}(T)$. Since the structure of $T$ is a tree, we can design a linear time algorithm to compute $i_{I}^{*}(T)$ based on dynamic programming.\\

\section{Algorithm}
\qquad  In the new algorithm, ten domination numbers will be given first. Given a connected block graph $G$, let $T$ be the block-cutpoint graph of $G$.  What we want to do is to find $i_{I}^{*}(T)$. During the process of  designing the algorithm, we just need to consider the IIIDF $f_{*}$ of $T$ which satisfies the five properties of
property \ref{PropertyFiveProperty} and $f_{*}(b) \in \{0,1\}$ where $b$ is a block1-vertex.\\

%Theorem \ref{ThmEquiRlationship} can promise the equivalent relationship between a block graph $G$ and its block-cutpoint graph $T$, then we can transfer the independent double Roman domination problem of $G$ to the induced independent double Roman domination problem of $T$. Theorem \ref{ExistPosiblilityOfMiDRDF-02} can promise the possibility of finding  $i_{idR}(T)=i_{dR}(G)$ on $T$ with $f_*(b) \in \{0,2\}$ for any block1-vertex $b \in T$.  Theorem \ref{ThmEquiRlationship} can also help us  distinguish whether $f_*$ is an IIDRDF or not. Since the block-cutpoint graph $T$ is a tree, then we can design the algorithm based on dynamic programming to solve the induced independent double Roman domination problem of $T$. We will firstly give some new domination numbers, and then three theorems will be given to promise the correctness of the algorithm. $u$ is a specific vertex of the block-cutpoint graph $T$.

\noindent \begin{tabular}{l}
    \specialrule{0em}{1.4pt}{1.4pt}
    %$i_{c}^{0}(T,u)=min\{w(f): f  is \ an \ IIDF \ of \ T \ with \ f(u)=0, u \ is \ a \ cut-vertex \}$ \\
    $i_{c}^{0}(T,u)=min\{w(f): f$ is an IIIDF of $T$ with $f(u)=0$, $u$ is a cut-vertex$ \}$ \\
    \specialrule{0em}{1.4pt}{1.4pt}
    $i_{c}^{1}(T,u)=min\{w(f): f$ is an IIIDF of $T$ with $f(u)=1$, $u$ is a cut-vertex$ \}$ \\
    \specialrule{0em}{1.4pt}{1.4pt}
    $i_{c}^{2}(T,u)=min\{w(f): f$ is an IIIDF of $T$ with $f(u)=2$, $u$ is a cut-vertex$ \}$ \\
    \specialrule{0em}{1.4pt}{1.4pt}
    $i_{c}^{00}(T,u)=min\{w(f): f$ is an IIIDF of $T-u$, $u$ is a cut-vertex $\}$ \\
    \specialrule{0em}{1.4pt}{1.4pt}
    $i_{c}^{01}(T,u)=min\{w(f_1): f$ is an IIIDF of $T$ with $f(u)=0$ and  $f_1$ is an IIIDF of $T+uw$ with \\ $f_{1}(w)=1$  and $f_1|_{T}=f$,
    $u$ is a cut-vertex,$w$ is a blokc1-vertex $\}$

   %%$ i_{c}^{01}(T,u)=min\{w(f):\ f \ is \ an \ IIDF \ of \ T+uw \ with \ f(u)=0, \ f(w)=1,$ \\
   %%$ u \ is \ a \ cut-vertex, \ w \ is \ a \ blokc1-vertex \} $

 \end{tabular}

%\noindent $i_{c}^{0}(T,u)=min\{\omega(f): f$ is an $IIDRDF$ of $T$ with $f(u)=0$, $u$ is a cut-vertex$ \}$ \\
%\noindent $i_{c}^{2}(T,u)=min\{\omega(f): f$ is an $IIDRDF$ of $T$ with $f(u)=2$, $u$ is a cut-vertex$ \}$ \\
%\noindent $i_{c}^{3}(T,u)=min\{\omega(f): f$ is an $IIDRDF$ of $T$ with $f(u)=3$, $u$ is a cut-vertex$ \}$ \\
%\noindent $i_{c}^{00}(T,u)=min\{\omega(f): f$ is an $IIDRDF$ of $T-u$, $u$ is a cut-vertex $\}$ \\
%\noindent $i_{c}^{02}(T,u)=min\{\omega(f): f$ is an $IIDRDF$ of $T+uw$ with $f(u)=0$,$f(w)=2$, $u$ is a cut-vertex,$w$ is a blokc1-vertex $\}$ \\

\noindent
\begin{tabular}{l}
  \specialrule{0em}{1.4pt}{1.4pt}
  $i_{b}^{0}(T,u)=min\{w(f): f$ is an IIIDF of $T$ with $f(N_{T}[u])=0$, $u$ is a block-vertex$\}$  \\
  \specialrule{0em}{1.4pt}{1.4pt}
  $i_{b}^{1}(T,u)=min\{w(f): f$ is an IIIDF of $T$ with $f(N_{T}[u])=1$, $u$ is a block-vertex$\}$ \\
  \specialrule{0em}{1.4pt}{1.4pt}
  $i_{b}^{2}(T,u)=min\{w(f): f$ is an IIIDF of $T$ with $f(N_{T}[u])=2$, $u$ is a block-vertex$\}$ \\
  \specialrule{0em}{1.4pt}{1.4pt}
  $i_{b}^{01}(T,u)=min\{w(f_1): f$ is an IIIDF of $T$ with $f(N_{T}[u])=0$ and $f_1$ is an IIIDF of $T + uw$ \\ with  $f_{1}(w)=1$ and $f_{1}|_{T}=f$, $u$  is a block-vertex,$w$ is a cut-vertex $\}$ \\
  \specialrule{0em}{1.4pt}{1.4pt}
  $i_{b}^{02}(T,u)=min\{w(f_1): f$ is an IIIDF of $T$ with $f(N_{T}[u])=0$ and $f_1$ is an IIIDF of $T + uw$ \\ with  $f_{1}(w)=2$ and $f_{1}|_{T}=f$, $u$  is a block-vertex,$w$ is a cut-vertex $\}$ \\
\end{tabular}\\

%\noindent $i_{b}^{0}(T,u)=min\{\omega(f): f$ is an $IIDRDF$ of $T$ with $f(N_{T}[u])=0$, $u$ is a block-vertex$\}$  \\
%\noindent $i_{b}^{2}(T,u)=min\{\omega(f): f$ is an $IIDRDF$ of $T$ with $f(N_{T}[u])=2$, $u$ is a block-vertex$\}$ \\
%\noindent $i_{b}^{3}(T,u)=min\{\omega(f): f$ is an $IIDRDF$ of $T$ with $f(N_{T}[u])=3$, $u$ is a block-vertex$\}$ \\
%\noindent $i_{b}^{02}(T,u)=min\{\omega(f): f$ is an $IIDRDF$ of $T+uw$ with $f(N_{T}[u])=0$, $f(w)=2$,$u$ is a block-vertex,$w$ is a cut-vertex $\}$ \\
%\noindent $i_{b}^{03}(T,u)=min\{\omega(f): f$ is an $IIDRDF$ of $T+uw$ with $f(N_{T}[u])=0$, $f(w)=3$,$u$ is a block-vertex,$w$ is a cut-vertex $\}$ \\

\begin{theorem}\label{TheoremJudgingidRG}
     $T$ is the block-cutpoint graph of a block graph $G$ and $u$ is a specific vertex of $T$. If $u$ is a cut-vertex, then $i_{I}(G)=min\{i_{c}^{0}(T,u), i_{c}^{1}(T,u), i_{c}^{2}(T,u)\}$. If $u$ is a block-vertex, then $i_{I}(G) = min\{i_{b}^{0}(T,u), i_{b}^{1}(T,u), i_{b}^{2}(T,u)\}$.
\end{theorem}
\begin{proof}
     Since the specific vertex $u$ is either a cut-vertex or a block-vertex, then the conclusion in the theorem is obviously correct.
\end{proof}

\begin{theorem}\label{ThmUnoinOfH and G-cut}
     Given two disjoint block-cutpoint graphs $G$ and $H$ with specific cut-vertex $u$ and block-vertex $v$ respectively. $T$ is a block-cutpoint graph with the specific vertex $u$, which is obtained from the disjoint union of $G$ and $H$ by joining a new edge $uv$. Then the following statements hold:

 \noindent \begin{tabular}{l}
     \specialrule{0em}{1.4pt}{1.4pt}
     $i_{c}^{0}(T,u)=min\{i_{c}^{0}(G,u)+i_{b}^{0}(H,v),
     i_{c}^{01}(G,u)+i_{b}^{1}(H,v)-1, i_{c}^{00}(G,u)+i_{b}^{2}(H,v)\}$ \\
     \specialrule{0em}{1.4pt}{1.4pt}
     $i_{c}^{1}(T,u)=i_{c}^{1}(G,u)+i_{b}^{01}(H,v)-1$ \\
     \specialrule{0em}{1.4pt}{1.4pt}
     $i_{c}^{2}(T,u)=i_{c}^{2}(G,u)+i_{b}^{02}(H,v)-2$ \\
     \specialrule{0em}{1.4pt}{1.4pt}
     $i_{c}^{00}(T,u)=i_{c}^{00}(G,u)+min\{i_{b}^{0}(H,v), i_{b}^{1}(H,v)+i_{b}^{2}(H,v)\}$ \\
     \specialrule{0em}{1.4pt}{1.4pt}
     $i_{c}^{01}(T,u)=min\{i_{c}^{01}(G,u)+i_{b}^{0}(H,v),i_{c}^{00}(G,u)+i_{b}^{1}(H,v)+1,
     i_{c}^{00}(G,u)+i_{b}^{2}(H,v)+1 \}$
 \end{tabular}
    %\noindent $i_{c}^{0}(T,u)=min\{i_{c}^{0}(G,u)+i_{b}^{0}(H,v),
    %i_{c}^{02}(G,u)+i_{b}^{2}(H,v)-2, i_{c}^{00}(G,u)+i_{b}^{3}(H,v)\}$ \\
    % $i_{c}^{2}(T,u)=i_{c}^{2}(G,u)+i_{b}^{02}(H,v)-2$ \\
    % $i_{c}^{3}(T,u)=i_{c}^{3}(G,u)+i_{b}^{03}(H,v)-3$ \\
    % $i_{c}^{00}(T,u)=i_{c}^{00}(G,u)+min\{i_{b}^{0}(H,v), i_{b}^{2}(H,v)+i_{b}^{3}(H,v)\}$ \\
    %$i_{c}^{02}(T,u)=min\{i_{c}^{02}(G,u)+i_{b}^{0}(H,v),i_{c}^{00}(G,u)+i_{b}^{2}(H,v)+2,
    % i_{c}^{00}(G,u)+i_{b}^{3}(H,v)+2 \}$\\
\end{theorem}
\begin{proof}
     Proving the statements in order: \\
     \noindent(\romannumeral1) Let $f$ be an IIIDF of $T$ with $f(u)=0$  and decompose $f$ into $f^{'}\cup f^{''}$ such that $f^{'}(u)=f(u)=0$, $\forall v_1 \in N_{H}[v], f^{''}(v_1)=f(v_1)$.  If $f^{''}(N_{H}[v])=f(N_{H}[v])=0$, then $f$ is a $i^{*}_{I}$-function of $T$ if and only if $f^{'}$ is a $i^{*}_{I}$-function of $G$ and $f^{''}$ is a $i^{*}_{I}$-function of $H$. If $f^{''}(N_{H}[v])=f(N_{H}[v])=1$, then $f$ is a $i^{*}_{I}$-function of $T$ if and only if $f^{''}$ is a $i^{*}_{I}$-function of $H$ and $f^{'}=f_1^{'}|_{G}$ where $f_1^{'}$ is a $i^{*}_{I}$-function of $G+uw$ with $f_1^{'}(w)=1$ and $w$ is a block1-vertex . If $f^{''}(N_{H}[v])=f(N_{H}[v])=2$, then $\exists w_0 \in N_{H}[v]$ such that $f^{''}(w_0)=2$. Hence, $f$ is a $i^{*}_{I}$-function of $T$ if and only if $f^{'}|_{G-u}$ is a $i^{*}_{I}$-function of $G-u$ and $f^{''}$ is a $i^{*}_{I}$-function of $H$.  \\

     \noindent(\romannumeral2) Let $f$ be an IIIDF of $T$ with $f(u)=1$ and decompose $f$ into $f^{'} \cup f^{''}$ such that $f^{'}(u)=f(u)=1, v_1\in N_{H}[v], f^{''}(v_1)=f(v_1)$. It is clear that $f^{''}(N_{H}[v])=f(N_{H}[v])=0$ for the independence of $f$. Therefore, $f$ is a $i^{*}_{I}$-function of $T$ if and only if $f^{'}$ is a $i^{*}_{I}$-function of $G$ and $f^{''}=f_2^{''}|_{H}$ where $f_2^{''}$ is a $i^{*}_{I}$-function of $H+vw$ with $f_{2}^{''}(1)=2$ and $w$ is a cut-vertex. \\

     \noindent(\romannumeral3) Let $f$ be an IIIDF of $T$ with $f(u)=2$ and decompose $f$ into $f^{'} \cup f^{''}$ such that $f^{'}(u)=f(u)=2, \forall v_1\in N_{H}[v], f^{''}(v_1)=f(v_1)$. It is clear that $f^{''}(N_{H}[v])=f(N_{H}[v])=0$ since the independence of $f$ . Therefore, $f$ is a $i^{*}_{I}$-function of $T$ if and only if $f^{'}$ is a $i^{*}_{I}$-function of $G$ and $f^{''}=f_2^{''}|_{H}$ where $f_2^{''}$ is a $i^{*}_{I}$-function of $H+vw$ with $f_{2}^{''}(w)=2$ and $w$ is a cut-vertex. \\

     \noindent(\romannumeral4)Let $f$ be an IIIDF of $T-u$ and decompose $f$ into $f^{'} \cup f^{''}$ such that $f^{'}$ is an IIIDF of $G-u$ and $f^{''}$ is an IIIDF of $H$. Therefore, $f$ is a $i^{*}_{I}$-function of $T$ if and only if $f^{'}$ is a $i^{*}_{I}$-function of $G-u$ and $f^{''}$ is a $i^{*}_{I}$-function of $H$. Since $i_{I}^{*}(H)=min\{i_b^{0}(H,v),
     i_{b}^{1}(H,v), i_b^{2}(H,v)  \}$, the corresponding conclusion in the theorem is correct. \\

      \noindent(\romannumeral5)Let $f$ be an IIIDF of $T+uw$ with $f(u)=0$ and $f(w)=1$ where $w$ is a block1-vertex, then $u$ is  a cut-vertex. Decompose $f$ into $f^{'} \cup f^{''}$ such that $ f^{'}(w)=f(w)=1, f^{'}(u)=f(u)=0,  \forall v_1 \in N_{H}[v], f^{''}(v_1)=f(v_1)$. If $f^{''}(N_{H}[v])=0$, then $f$ is a $i^{*}_{I}$-function of $T$ if and only if $f^{'}$ is a $i^{*}_{I}$-function of $G+uw$ and $f^{''}$ is a $i^{*}_{I}$-function of $H$. If $f^{''}(N_{H}[v])=1$, then $\sum_{v_2 \in N_{T+uw}^{2}[u]}f(v_2)=2$. Therefore, $f$ is a $i^{*}_{I}$-function of $T$ if and only if $f^{'}|_{G-u}$ is a $i^{*}_{I}$-function of $G-u$ with $f^{'}(w)=1$ and $f^{''}$ is a $i^{*}_{I}$-function of $H$. If $f^{''}(N_{H}[v])=2$, it's obvious that  $\sum_{v_2 \in N_{T+uw}^{2}[u]}f(v_2)=3$ . Therefore, $f$ is a $i^{*}_{I}$-function of $T$ if and only if $f^{'}|_{G-u}$ is a $i^{*}_{I}$-function of $G-u$ with $f^{'}(w)=1$ and $f^{''}$ is a $i^{*}_{I}$-function of $H$. \\
\end{proof}

\begin{theorem}\label{ThmUnoinOfH and G-block}
      Given two disjoint block-cutpoint graphs $G$ and $H$ with specific block-vertex $u$ and cut-vertex $v$ respectively. $T$ is a graph with the specific vertex $u$, which is obtained from the disjoint union of $G$ and $H$ by joining a new edge $uv$. Then the following statements hold:
\noindent
\begin{tabular}{l}
   \specialrule{0em}{1.4pt}{1.4pt}
   $i_{b}^{0}(T,u)=i_{b}^{0}(G,u) + i_{c}^{0}(H,v)$ \\
   \specialrule{0em}{1.4pt}{1.4pt}
   $i_{b}^{1}(T,u)=min\{i_{b}^{1}(G,u)+i_{c}^{01}(H,v)-1,
      i_{b}^{01}(G,u)+i_{c}^{1}(H,v)-1\}$ \\
   \specialrule{0em}{1.4pt}{1.4pt}
   $i_{b}^{2}(T,u)=min\{i_{b}^{2}(G,u)+i_{c}^{00}(H,v),
      i_{b}^{02}(G,u)+i_{c}^{2}(H,v)-2 \}$ \\
   \specialrule{0em}{1.4pt}{1.4pt}
   $i_{b}^{01}(T,u)=i_{b}^{01}(G,u)+i_{c}^{01}(H,v)-1$ \\
   \specialrule{0em}{1.4pt}{1.4pt}
   $i_{b}^{02}(T,u)=i_{b}^{02}(G,u) + i_{c}^{00}(H,v)$\\
\end{tabular}

%\noindent $i_{b}^{0}(T,u)=i_{b}^{0}(G,v) + i_{c}^{0}(H,v)$ \\
%\noindent $i_{b}^{2}=min\{i_{b}^{2}(G,u)+i_{c}^{02}(H,v)-2,
%i_{b}^{02}+i_{c}^{2}(H,v)-2\}$ \\
%\noindent $i_{b}^{3}=min\{i_{b}^{3}(H,v)+i_{c}^{00}(H,v),
%i_{b}^{03}(G,u)+i_{c}^{3}(H,v)-3 \}$ \\
%\noindent
%$i_{b}^{02}(T,u)=i_{b}^{02}(G,u)+i_{c}^{02}(H,v)-2$ \\
%\noindent
%$i_{b}^{03}(T,u)=i_{b}^{03}(G,u) + i_{c}^{00}(H,v)$
\end{theorem}
\begin{proof}
     It is obvious that $N_{T}[u]=N_{G}[u] \cup \{v\}$, proving the statements in order: \\
     \noindent(\romannumeral1)Let $f$ be an IIIDF of $T$ with $f(N_{T}[u])=0$, then $f(N_{G}[u])=0$, $f(v)=0$. Decompose $f$ into $f^{'}\cup f^{''}$ such that $\forall u_1 \in N_{H}[u], f^{'}(u_1)=f(u_1)$, $f^{''}(v)=f(v)$. Hence, $f^{'}(N_{G}[u])=f(N_{G}[u])=0$. Therefore, $f$ is a $i^{*}_{I}$-function of $T$ if and only if $f^{'}$ is a $i^{*}_{I}$-function of $G$ and $f^{''}$ is a $i^{*}_{I}$-function of $H$. \\

     \noindent(\romannumeral2)Let $f$ be an IIIDF of $T$ with $f(N_{T}[u])=1$. Decompose $f$ into $f^{'} \cup f^{''}$ such that $\forall u^{'} \in N_{G}[u], f^{'}(u^{'})=f(u^{'})$, $f^{''}(v)=f(v)$. Since $N_{T}[u]=N_{G}[u] \cup \{v\}$, $f^{'}(N_{G}[u])=f(N_{G}[u])=0$ if $f^{''}(v)=f(v)=1$. Therefore, $f$ is a $i^{*}_{I}$-function of $T$ if and only if $f^{''}$ is a $i^{*}_{I}$-function of $H$ and $f^{'}=f^{'}_{1}|_{G}$ where
     $f_{1}^{'}$ is a $i^{*}_{I}$-function of $G+uw_1$ and $w_1$ is a cut-vertex with $f_1^{'}(w_1)=1$. $f^{'}(N_{G}[u])=f(N_{G}[u])=1$ if $f^{''}(v)=f(v)=0$. Hence, $f$ is a $i^{*}_{I}$-function of $T$ if and only if $f^{'}$ is a $i^{*}_{I}$-function of $G$ and $f^{''}=f^{''}_{2}|_{H}$ where $f^{''}_{2}$ is a $i^{*}_{I}$-function of $H+vw_2$ satisfying $w_2$ is a block1-vertex and $f_{2}^{''}(w_2)=1$. \\

     \noindent(\romannumeral3)Let $f$ be an IIIDF  of $T$ with $f(N_{T}[u])=2$. Decompose $f$ into $f^{'} \cup f^{''}$ such that $\forall u^{'} \in N_{G}[u], f^{'}(u^{'})=f(u^{'})$, $f^{''}(v)=f(v)$, then either $f(N_{G}[u])=2,f(v)=0$ or $f(N_{G}[u])=0,f(v)=2$. Since $N_{T}[u]=N_{G}[u] \cup \{v\}$,
     $f^{'}(N_{G}[u])=f(N_{G}[u])=0$ if $f^{''}(v)=f(v)=2$. Therefore, $f$ is a $i^{*}_{I}$-function of $T$ if and only if $f^{''}$ is a $i^{*}_{I}$-function of $H$ and $f^{'}=f_{1}^{'}|_{G}$ where $f_1^{'}$ is a $i^{*}_{I}$-function of $G+uw_1$ satisfying $w_1$ is a cut-vertex and $f_1^{'}(w)=2$. $f^{'}(N_{G}[u])=f(N_{G}[u])=2$ if $f^{''}(v)=f(v)=0$, then $\sum_{v_0 \in N_{T}^{2}[v]}f(v_0)=2$. Therefore, $f$ is a $i^{*}_{I}$-function of $T$ if and only if $f^{'}$ is a $i^{*}_{I}$-function of $G$ and $f^{''}|_{H-v}$ is a $i^{*}_{I}$-function of $H-v$. \\

     \noindent(\romannumeral4)Let $f$ be an IIIDF of $T+uw$ with $f(w)=1$ and  $f(N_{T}[u])=0$. Obviously, $w$ is a cut-vertex, then we can get that  $f(N_{G}[u])=0$ and $f(v)=0$.  Decompose $f$ into $f^{'} \cup f^{''}$ such that
     $f^{'}(N_{G}[u])=f(N_{G}[u])=0, f^{'}(w)=f(w)=1, f^{''}(v)=f(v)=0$, then $f$ is a $i^{*}_{I}$-function of $T+uw$ if and only if $f^{'}$ is a $i^{*}_{I}$-function of $G+uw$ and $f^{''}=f_{2}^{''}|_{H}$ where $f_{2}^{''}$ is a $i^{*}_{I}$-function of $H+vw_2$ with $f_{2}^{''}(w_2)=1$ and $w_2$ is a block1-vertex. \\

     \noindent(\romannumeral5) Let $f$ be an IIIDF of $T+uw$ with $f(N_{T}[u])=0, f(w)=2$. Obviously, $w$ is a cut-vertex and we get that $f(N_{G}[u])=0$ and $f(v)=0$.  Decompose $f$ into $f^{'} \cup f^{''}$ such that $f^{'}(N_{G}[u])=f(N_{G}[u])=0$,$f^{'}(w)=f(w)=2$, $f^{''}(v)=f(v)=0$, getting $\sum_{v_1 \in N_{T+uw}[v]}f(v_1) \geq 2$. Therefore, $f$ is a $i^{*}_{I}$-function of $T+uw$ if and only if $f^{'}$ is a $i^{*}_{I}$-function of $G+uw$ and $f^{''}=f_{2}^{''}|_{H-v}$ where $f_{2}^{''}$ is a $i^{*}_{I}$-function of $H+vw_2$ with $f_{2}^{''}(w_2)=2$ and $w_2$ is a block1-vertex.  \\
\end{proof}

We have finished proving theorem \ref{TheoremJudgingidRG}, theorem \ref{ThmUnoinOfH and G-cut} and theorem \ref{ThmUnoinOfH and G-block}. A new algorithm will be designed based on the three algorithms to output the independent domination number of any connected block graph. The algorithm is designed based on dynamic programming \cite{dynamicProgramming}.   The correctness of the algorithm can be promised by the three theorems. As for the initialization of the ten domination numbers, we just need to consider the domination numbers of the block-cutpoint graph $T$ where $T$ is only one vertex. The domination number will be initialized as $\infty$ if it does not exist.   Therefore, [$i_c^{0}(T,u)$, $i_c^{1}(T,u)$, $i_c^{2}(T,u)$, $i_c^{00}(T,u)$,$i_c^{01}(T,u)$] can be initialized as [$\infty,1,2,0,\infty$]. Initialize [$i_b^{0}(T,u)$, $i_b^{1}(T,u)$, $i_b^{2}(T,u)$, $i_b^{01}(T,u)$,$i_b^{02}(T,u)$] as [$0,\infty,\infty,1,2$] if $T$ is a block0-vertex. Initialize [$i_b^{0}(T,u)$, $i_b^{1}(T,u)$, $i_b^{2}(T,u)$, $i_b^{01}(T,u)$,$i_b^{02}(T,u)$] as [$\infty,1,\infty,\infty,2$] if $T$ is a block1-vertex. Initialize  [$i_b^{0}(T,u)$, $i_b^{1}(T,u)$, $i_b^{2}(T,u)$, $i_b^{01}(T,u)$,$i_b^{02}(T,u)$] as [$\infty,\infty,2,\infty,2$] if $T$ is a block2-vertex. The new algorithm will be given below to output the $i_{I}(G)$ of any connected block graph $G$, in which $T$ is the block-cutpoint graph of $G$. \\

\IncMargin{1em}  %使得行号不向外突出
\begin{algorithm} \label{AlgorithmOnBlockGraph}
   \setstretch{1.2}
   %\SetAlgoNoLine  %不要算法中的竖线
   \SetKwInOut{Input}{\textbf{Input}}
   \SetKwInOut{Output}{\textbf{Output}} %替换关键词
   \Input{Tree order $[v_1,v_2,...,v_n]$ of the block-cutpoint graph $T$}
   \Output{The independent Italian domination number $i_{I}(G)$}
   \BlankLine
  \uIf{$G=K_1$}{\textbf{return}$~~i_{I}(G)=1$;}
  \uIf{$G=K_m(m\geq2)$}{\textbf{return}$~~i_{I}(G)=2$;}
  $initialization:$ \\
  \For{$i=1$ to $n$}{
     \uIf{$v_i$ is a cut-vertex}{
     [$i_{c}^{0}(v_i),i_{c}^{1}(v_i), i_{c}^{2}(v_i),i_{c}^{00}(v_i),i_{c}^{01}(v_i)$] $\leftarrow$ [$\infty,1,2,0,\infty$];
     }
     \uElseIf{$v_i$ is a block0-vertex }{
     [$i_{b}^{0}(v_i),i_{b}^{1}(v_i),i_{b}^{2}(v_i),
     i_{b}^{01}(v_i),i_{b}^{02}(v_i)$ ]$\leftarrow$ [$0,\infty,\infty,1,2$];
     }
     \uElseIf{$v_i$ is a block1-vertex}{
     [$i_{b}^{0}(v_i),i_{b}^{1}(v_i) ,i_{b}^{2}(v_i),
     i_{b}^{01}(v_i),i_{b}^{02}(v_i)$ ]$\leftarrow$ [$\infty,1,\infty,\infty,2$];
     }
     \uElse{
     [$i_{b}^{0}(v_i),i_{b}^{1}(v_i),i_{b}^{2}(v_i),
     i_{b}^{01}(v_i),i_{b}^{02}(v_i)$] $\leftarrow $ [$\infty,\infty,2,\infty,2$].}
  }

  \For{$i=1$ to $n-1$}{
      let $v_j$ be the parent of $v_i$;  \\
      \uIf{$v_i$ is a cut-vertex}{
         $i_{b}^{0}(v_j) \leftarrow i_{b}^{0}(v_j) + i_{c}^{0}(v_i)$; \\
         $i_{b}^{1}(v_j) \leftarrow min\{i_{b}^{1}(v_j)+i_{c}^{01}(v_i)-1,i_{b}^{01}(v_j)+i_{c}^{1}(v_i)-1\}$; \\
         $i_{b}^{2}(v_j) \leftarrow min\{i_{b}^{2}(v_j)+i_{c}^{00}(v_i),i_{b}^{02}(v_j)+i_{c}^{2}(v_i)-2 \}$; \\
         $i_{b}^{01}(v_j)\leftarrow i_{b}^{01}(v_j)+i_{c}^{01}(v_i)-1$; \\
         $i_{b}^{02}(v_j)\leftarrow i_{b}^{02}(v_j)+i_{c}^{00}(v_i)$. \\
      }
      \uElse{
       $i_{c}^{0}(v_j) \leftarrow min\{i_{c}^{0}(v_j)+i_{b}^{0}(v_i),i_{c}^{01}(v_j)+i_{b}^{1}(v_i)-1,i_{c}^{00}(v_j)+i_{b}^{2}(v_i)\}$; \\
       $i_{c}^{1}(v_j) \leftarrow i_{c}^{1}(v_j) + i_{b}^{01}(v_i)-1 $;\\
       $i_{c}^{2}(v_j) \leftarrow i_{c}^{2}(v_j)+ i_{b}^{02}(v_i)-2$; \\
       $i_{c}^{00}(v_j)\leftarrow i_{c}^{00}(v_j) + min\{i_{b}^{0}(v_i),i_{b}^{1}(v_j),i_{b}^{2}(v_i)\}$; \\
       $i_{c}^{01}(v_j)\leftarrow min\{i_{c}^{01}(v_j)+i_{b}^{0}(v_i),i_{c}^{00}(v_j)+i_{b}^{1}(v_i)+1,i_{c}^{00}(v_j)+i_{b}^{2}(v_i) +1\}$. \\
      }
  }

  \uIf{$v_n$ is a cut-vertex}{\textbf{return}~$i_{I}(G)=min\{i_{c}^{0}(v_n)$,$i_{c}^{1}(v_n)$,$i_{c}^{2}(v_n)$\}; }
  \Else{\textbf{return}~$i_{I}(G)= min\{i_{b}^{0}(v_n)$,$i_{b}^{1}(v_n)$,$i_{b}^{2}(v_n)$\}.}
   ~~\\

\caption{Independent Italian Domination on Block Graph}
\end{algorithm}
\DecMargin{1em}

\begin{theorem}
     Given an arbitrary connected block graph  $G$ and its corresponding block-cutpoint graph $T=(V,E)$ with tree order $[v_1,v_2,...,v_n]$. Algorithm \textbf{\ref{AlgorithmOnBlockGraph}} can output the independent Italian domination number $i_{I}(G)$ of  $G$ in linear time $O(n+m)$ where $n=|V|$ and $m=|E|$.
\end{theorem}

\end{spacing}
\end{document}